\documentclass[12pt]{amsart}

\usepackage{extsizes}
\usepackage{amssymb, amsmath}
\usepackage{fourier}
\usepackage{color}

\usepackage{geometry}
\theoremstyle{plain}
\newtheorem{theorem}{Theorem}

\newtheorem{lemma}[theorem]{Lemma}

\theoremstyle{remark}

\newcommand{\cmp}{\mathsf{m}_p}
\newcommand{\kamp}{\mathfrak{m}_p}
\newcommand{\ep}{\frac{1}{p}}

\newcommand{\A}{\mathcal A}

\newcommand{\Ap}{\mathcal A^+}
\newcommand{\App}{\mathcal A^{++}}
\newcommand{\B}{\mathcal B}

\newcommand{\Bp}{\mathcal {B}^+}
\newcommand{\Bpp}{\mathcal B^{++}}

\newcommand{\bhp}{B(H)^+}
\newcommand{\bhpp}{B(H)^{++}}
\newcommand\rng{\operatorname{rng}}

\mathchardef\hy="2D

\renewcommand{\l}{\lambda}

\newcommand{\ler}[1]{\left( #1 \right)}

\newcommand{\fel}{\frac{1}{2}}
\newcommand{\mfel}{{-\frac{1}{2}}}

\begin{document}

\title[]{Maps on positive cones in operator algebras preserving power means}

\author{Lajos Moln\'ar}
\address{University of Szeged, Interdisciplinary Excellence Centre, Bolyai Institute,
H-6720 Szeged, Aradi v\'ertan\'uk tere 1.,
Hungary, and
Budapest University of Technology and Economics,  Institute of Mathematics,
H-1521 Budapest, Hungary}
\email{molnarl@math.u-szeged.hu}
\urladdr{http://www.math.u-szeged.hu/\~{}molnarl}

\dedicatory{Dedicated to Professor J\'anos Acz\'el on the occasion of his 95th birthday}

\thanks{
This paper was written while the author was a visiting researcher at the Alfr\'ed R\'enyi Institute of Mathematics (Hungarian Academy of Sciences).
Ministry of Human Capacities, Hungary grant 20391-3/2018/FEKUSTRAT is  acknowledged and the work was also supported by the National Research, Development and Innovation Office of Hungary, NKFIH, Grant No. K115383.}

\begin{abstract}
In this paper we consider power means of positive Hilbert space operators both in the conventional and in the Kubo-Ando senses. We describe the corresponding isomorphisms (bijective transformations respecting those means as binary  operations) on positive definite cones and on positive semidefinite cones in operator algebras. We also investigate the question when those two sorts of power means can be transformed into each other.
\end{abstract}

\maketitle

\section{Introduction and statements of the results} 

Let $p$ be a nonzero real number. The $p$th power mean $M_p(t,s)$ of the positive real numbers $t,s$ is defined by
\begin{equation}\label{E:4}
M_p(t,s)=\ler{\frac{t^p+s^p}{2}}^\ep
\end{equation}
and it is one of the most fundamental kinds of means for numbers. Means of positive (definite or semidefinite) matrices or Hilbert space operators are also very important concepts having wide range of applications.
Formula \eqref{E:4} above can easily be extended to that setting at least in the positive definite case. In fact, using functional calculus, formally the same definition works fine, we just write positive definite matrices or operators, or positive definite elements $A,B$ of a $C^*$-algebra in the place of the numbers $t,s$ in \eqref{E:4}:
\[
\ler{\frac{A^p+B^p}{2}}^\ep.
\]
The so obtained concept that can be called the conventional $p$th power mean is, though very natural, not really satisfactory for the purposes of matrix theory and operator theory. Indeed, the main disadvantage is that it is not monotone in its variables with respect to the usual, L\"owner order, which comes from the concept of positive semidefiniteness. The 'proper' definition of the power mean in the case where $p\in ]-1,1[, p\neq 0$ is the following one:
\[
A^\fel\ler{\frac{I+(A^\mfel B A^\mfel)^p}{2}}^\ep A^\fel.
\]
Let us explain this in a bit more details.

Probably the most important notion of means for positive semidefinite matrices or Hilbert operators is due to Kubo and Ando \cite{KubAnd80}. Very briefly, it can be summarized as follows. Let $H$ be a complex Hilbert space. Denote by $\bhp$ the convex cone of all bounded positive semidefinite linear operators on $H$. We say that a binary operation $\sigma$ on $\bhp$ is a Kubo-Ando mean if  the following requirements are fulfilled (from (a) to (d), all operators are supposed to belong to $\bhp$):

\begin{itemize}
\item[(a)] $I\sigma I=I$;
\item[(b)] if $A\leq C$ and $B\leq D$, then $A\sigma B\leq C\sigma D$;
\item[(c)] $C(A\sigma B)C\leq (CAC)\sigma (CBC)$;
\item[(d)] if $A_n \downarrow A$ and $B_n \downarrow B$ strongly,
then $A_n\sigma B_n \downarrow A\sigma B$ strongly (the sign $\downarrow$ refers to monotone decreasing convergence in the usual (L\"owner) order among self-adjoint operators).
\end{itemize}
Convex combination among Kubo-Ando means is defined in the straightforward, natural way. The celebrated result, Theorem 3.2 in \cite{KubAnd80} says that, for infinite dimensional $H$, there is an affine isomorphism from the class of all Kubo-Ando means $\sigma$ on $\bhp$ onto the class of all operator monotone functions $f:]0,\infty [\to ]0,\infty[$ with the property $f(1)=1$ which is given by the formula $f(t)I=I\sigma tI$, $t>0$. For invertible $A,B\in \bhp$, we have
\begin{equation}\label{E:funct}
A\sigma B=A^{\fel} f(A^\mfel BA^\mfel) A^{\fel}.
\end{equation}
Observe that the theorem above implies that Kubo-Ando means do not depend on the underlying Hilbert spaces, they only depend on certain (very special) real functions.
By property (d) we obtain that formula \eqref{E:funct} extends to any invertible $A\in \bhp$ and arbitrary $B\in \bhp$.
The most distinguished Kubo-Ando means are naturally the arithmetic mean with representing function $t\mapsto (1+t)/2$, the harmonic mean with representing function $t\mapsto (2t)/(1+t)$ and the geometric mean with representing function $t\mapsto \sqrt t$, $t>0$. For invertible $A,B\in \bhp$, those means of $A,B$ are in turn the following operators
\[
\frac{A+B}{2}, \quad 2(A^{-1}+B^{-1})^{-1},  \quad A^{\fel} (A^\mfel BA^\mfel)^{\fel} A^{\fel}.
\]
Below, whenever we write $\sigma, f$, we always mean that $\sigma$ is a Kubo-Ando mean and $f$ is its representing operator monotone function. We know from the deep theory of operator monotone functions that each such $f$ has a holomorphic extension to the complex upper half plane. The transpose $\sigma'$ of the Kubo-Ando mean $\sigma$ is the mean
satisfying $A\sigma' B=B\sigma A$, $A,B\in \bhp$. Its representing function is $t\mapsto tf(1/t)$, $t>0$. The Kubo-Ando mean $\sigma$ is called symmetric if $\sigma'=\sigma$.
The adjoint $\sigma^*$ of $\sigma$ is the Kubo-Ando mean satisfying $A\sigma^* B=(A^{-1}\sigma B^{-1})^{-1}$ for all invertible $A,B\in \bhp$. Its representing function is $t\mapsto 1/f(1/t)$, $t>0$. Clearly, the arithmetic, harmonic and geometric means are symmetric Kubo-Ando means and the former two are the adjoints of each other.

In what follows, we introduce some notation. If $\A$ is a $C^*$-algebra (in this paper all $C^*$-algebras are assumed to be unital), then $\Ap$ stands for the set of all positive semidefinite elements of $\A$, i.e. elements which are self-adjoint and have nonnegative spectrum. It is called the positive semidefinite cone of $\A$. The subset $\App$ of $\Ap$  containing the positive definite elements, i.e. the invertible elements in $\Ap$, is called the positive definite cone of $\A$. 

As already referred to above, for any nonzero $p$, we define the conventional $p$th power mean $\cmp$ on $\App$ by the formula
\begin{equation*}\label{E:CON}
A\cmp B=\ler{\frac{A^p+B^p}{2}}^\ep, \quad A,B\in \App.
\end{equation*}
The Kubo-Ando $p$th power mean $\kamp$ corresponds to the operator monotone function
\[
t\mapsto \ler{\frac{1+t^p}{2}}^\ep, \quad t>0.
\]
In fact, it is known that for the operator monotonicity of this function we need to assume that  $p\in ]-1,1[$, see Theorem 4 in \cite{BesPet}. So, for such $p$, we define
\begin{equation}\label{E:KA}
A\kamp B=A^\fel\ler{\frac{I+(A^\mfel B A^\mfel)^p}{2}}^\ep A^\fel
, \quad A,B\in \App.
\end{equation}
It can be easily seen that for commuting $A,B\in \App$ we have
\begin{equation}\label{E:com}
A\kamp B=A\cmp B.
\end{equation}

In this paper we discuss maps respecting the operation of power means. Similar studies were made for the arithmetic and harmonic means in \cite{ML09f} and for the geometric mean in \cite{ML09c} on the positive semidefinite cone $\bhp$ . Related results concerning the positive definite cones in $C^*$-algebras were presented in \cite{MolJim}. Propositions 1 and 2 in \cite{MolJim} describe all bijective maps between positive definite cones in $C^*$-algebras which preserve either the arithmetic or the harmonic mean, while Theorem 4 in \cite{MolJim} gives the precise structure of continuous bijective maps between the positive definite cones in von Neumann factors which preserve the geometric mean. The results which we present here are similar to those in the sense that also here it turns out that the transformations under considerations are closely related to the so-called    Jordan *-isomorphisms between the underlying algebras.

Real functions respecting means of real numbers were considered in a number of papers under the name 'mean affine functions'. We mention only a few of them: \cite{Berr}, \cite{Mat}, \cite{Nowak}. The concept appears also in the fundamental book \cite{AD} on functional equations, see page 251. The problems we consider here are clearly related to those investigations but our setting here is much more complicated.

In our first three results we deal with power mean respecting maps in the context of $C^*$-algebras. As for the conventional power mean, we have the following description of the corresponding maps.

\begin{theorem}\label{T:2}
Let $p$ be a nonzero real number, $\A,\B$ be $C^*$-algebras and let $\phi:\App \to \Bpp$ be a bijective map. Then $\phi$
satisfies
\begin{equation*}
\phi(A\cmp B)=\phi(A)\cmp \phi(B), \quad A,B\in \App
\end{equation*}
if and only if there are a Jordan *-isomorphism $J:\A\to \B$ and an element $D\in \Bpp$ such that 
\[
\phi(A)=(DJ(A)^pD)^\ep,\quad A\in \App.
\]
\end{theorem}

The concept of Jordan *-isomorphisms (or, in other words, $C^*$-isomorphisms) that appears here is of fundamental importance in the theory of operator algebras for many reasons. In a certain sense they are the most important sorts of isomorphisms between those structures. The bijective map $J:\A\to \B$ between $*$-algebras $\A$ and $\B$ is called a Jordan *-isomorphism if it is linear, preserves the Jordan product, i.e.,
\[
J(AB+BA)=J(A)J(B)+J(B)J(A), \quad A,B\in \A,
\]
and preserves also the *-operation, i.e.,
\[
J(A^*)=J(A)^*, \quad A\in \A.
\]
The point of the theorem above is that, as it shows, if a bijective map between the positive definite cones of $C^*$-algebras preserves a conventional power mean, then it is necessarily closely related to a bijective linear transformation between the underlying algebras which is a kind of multiplicative (with respect to the Jordan product).

In the next result we obtain a description of the Kubo-Ando power mean preserving bijections of positive definite cones. However, here we need to assume the continuity of the transformations. 

\begin{theorem}\label{T:1}
Let $p\in [-1,1]$ be a nonzero real number, $\A,\B$ be $C^*$-algebras and let $\phi:\App \to \Bpp$ be a continuous bijective map. Then $\phi$
satisfies
\begin{equation*}
\phi(A\kamp B)=\phi(A)\kamp \phi(B), \quad A,B\in \App
\end{equation*}
if and only if there are a Jordan *-isomorphism $J:\A\to \B$ and an element $D\in \Bpp$ such that 
\[
\phi(A)=DJ(A)D,\quad A\in \App.
\]
\end{theorem}

It is interesting to observe that, for $p\neq \pm 1$, the groups of (continuous) automorphisms of the operations $\cmp$ and $\kamp$ are different, which implies that the operations themselves are also quite different in general.   
But on the positive definite cones in commutative $C^*$-algebras, the conventional and Kubo-Ando power means obviously coincide. We will prove that the converse is also true. In fact, we will show the much stronger result that if one positive definite cone equipped with the conventional power mean is isomorphic (via a continuous bijection) to another positive definite cone equipped with the Kubo-Ando power mean, then the underlying algebras are necessarily commutative. The precise statement reads as follows.

\begin{theorem}\label{T:3}
Let $p\in]-1,1[$ be a nonzero real number and let  $\A,\B$ be $C^*$-algebras. Assume that $\phi:\App \to \Bpp$ is a continuous bijective map such that
\begin{equation*}
\phi(A\cmp B)=\phi(A)\kamp \phi(B), \quad A,B\in \App.
\end{equation*}
Then the algebras $\A,\B$ are necessarily commutative.
\end{theorem}

We now turn to the case of positive semidefinite cones. In the remaining results,  let $H$ be a complex Hilbert space.
Since $\kamp$ is a Kubo-Ando mean, $A\kamp B$ is defined for any pairs of elements of $\bhp$. As for the conventional power mean $\cmp$, it is no problem to define it on $\bhp$ when $p$ is positive. But how to define it for negative $p$? We can handle the case $-1\leq p<0$ as follows. Observe that with $q=-p$, for positive invertible $A,B\in \bhp$ we have
\[
A\cmp B=\ler{\frac{A^{-q}+B^{-q}}{2}}^{-\frac{1}{q}}=(A^q \mathfrak{m}_{-1}  B^q)^{\frac{1}{q}}.
\]
If $A,B \in \bhp$ are arbitrary and $A_n,B_n\in \bhp$ are invertible, they form norm bounded sequences for which $A_n\downarrow A$, $B_n\downarrow B$ in the strong operator topology, then we have the same type of convergence for the $q$th powers. Indeed, this follows from the operator monotonicity of the function $t\mapsto t^q$ and from the fact that any continuous bounded real function is strongly continuous (see 4.3.2. Theorem in \cite{Mur}). By the property (d) of Kubo-Ando means, we deduce that $A_n^q \mathfrak{m}_{-1}  B_n^q \downarrow A^q \mathfrak{m}_{-1}  B^q$ strongly and then that $(A_n^q \mathfrak{m}_{-1}  B_n^q)^{\frac{1}{q}}$ converges to $(A^q \mathfrak{m}_{-1}  B^q)^{\frac{1}{q}}$  strongly (however, monotonicity is no longer guaranteed since the exponent $1/q$ is greater than or equal to 1, and the power function with exponent greater than 1 is well-known to be not operator monotone). Therefore, for $-1\leq p<0$, it is natural to define
\begin{equation}\label{E:8}
A\cmp B=(A^q \mathfrak{m}_{-1}  B^q)^{\frac{1}{q}}, \quad A,B\in \bhp.
\end{equation}

We can now formulate our result concerning maps on the positive semidefinite cone $\bhp$ which preserve the conventional power mean. We highlight the fact that neither in relation with $\cmp$, nor in relation with $\kamp$ do we assume the continuity of the transformations under consideration. This is a serious difference between the cases of the positive definite and semidefinite cones (see the former two results). 

\begin{theorem}\label{T:5}
Let $p\in [-1,1]$ be a nonzero real number, $\phi:\bhp \to \bhp$ be a bijective map such that
\begin{equation}\label{E:2}
\phi(A\cmp B)=\phi(A)\cmp \phi(B), \quad A,B\in \bhp.
\end{equation}
Then there is an invertible bounded either linear or conjugate linear operator $T:H\to H$ such that $\phi$ is of the form
\begin{equation}\label{E:5}
\phi(A)=(TA^{|p|}T^*)^{\frac{1}{|p|}}, \quad A\in \bhp.
\end{equation} 
Conversely, every map $\phi :\bhp \to \bhp$ of the form \eqref{E:5} satisfies \eqref{E:2}.
\end{theorem}

The next result concerns the Kubo-Ando power mean.

\begin{theorem}\label{T:4}
Let $p\in [-1,1]$ be a nonzero real number,
$\phi:\bhp \to \bhp$ be a bijective map such that
\begin{equation}\label{E:7}
\phi(A\kamp B)=\phi(A)\kamp \phi(B), \quad A,B\in \bhp.
\end{equation}
Then there is an invertible bounded either linear or conjugate linear operator $T:H\to H$ such that $\phi$ is of the form
\begin{equation}\label{E:6}
\phi(A)=TAT^*, \quad A\in \bhp.
\end{equation} 
Conversely, every map $\phi :\bhp \to \bhp$ of the form \eqref{E:6} satisfies \eqref{E:7}.
\end{theorem}

Finally, concerning the existence of a map transforming the conventional power mean to the Kubo-Ando power mean, we have the following statement.

\begin{theorem}\label{T:6}
Let $p\in ]-1,1[$ be a nonzero real number and $H$ be a Hilbert space of dimension at least 2. Then
there is no bijective map $\phi:\bhp \to \bhp$ such that
\begin{equation}\label{E:3}
\phi(A\cmp B)=\phi(A)\kamp \phi(B), \quad A,B\in \bhp.
\end{equation}
\end{theorem}

\section{Proofs}

In this section we present the proofs of our results. We begin with verifying the statements concerning power mean preservers on the positive definite cones of $C^*$-algebras. But first let us summarize some of the basic properties of Jordan *-isomorphisms between $C^*$-algebras that we will need in what follows.

Let $\mathcal A, \mathcal B$ be $C^*$-algebras and $J:\mathcal A\to \mathcal B$ be a Jordan *-isomorphism. It is apparent that $J$ is a linear order isomorphism between the self-adjoint parts of $\mathcal A$ and $\mathcal B$. Moreover, $J$ is an isometry
\begin{equation*}\label{E:J0}
\|J(X)\|=\|X\|, \quad X\in \mathcal A,
\end{equation*}
see, e.g., \cite{Mat05}. Next, $J$ satisfies
\begin{equation*}\label{E:J1}
J(XYX)=J(X)J(Y)J(X), \quad X,Y\in \mathcal A,
\end{equation*}
and hence
\begin{equation}\label{E:J2}
J(X^n)=J(X)^n, \quad X\in \mathcal A
\end{equation}
holds for every nonnegative integer $n$, see 6.3.2 Lemma in \cite{Pal}. In particular, $J$ is unital meaning that $J$ sends the identity to the identity.  
By Proposition 1.3 in \cite{Sou}, $J$ preserves invertibility,  namely we have
\begin{equation*}\label{E:J3}
J(X^{-1})=J(X)^{-1}
\end{equation*}
for every invertible element $X\in \mathcal A$.
It follows that $J$ preserves the spectrum and, using continuous function calculus, from \eqref{E:J2} we deduce that 
\begin{equation*}\label{E:J4}
J(f(X))=f(J(X))
\end{equation*}
holds for any self-adjoint element $X\in \mathcal A$ and continuous real function $f$ on the spectrum of $X$.  
It then follows that $J$ maps $\mathcal A^{++}$ onto $\mathcal B^{++}$ and
$J$ is an isomorphism between $\App$ and $\Bpp$ with respect to any Kubo-Ando means.

After this, we start with the easy proof of our first result.

\begin{proof}[Proof of Theorem \ref{T:2}]
Assume first that $p=1$. In that case the result is exactly Proposition 1 in \cite{MolJim}. If $p\neq 1$, then consider the map $\psi:\App\to \Bpp$ defined by $\psi(A)=\phi(A^\ep)^p$, $A\in \App$. It is easy to see that $\psi$ is a bijective map between the positive definite cones $\App$ and $\Bpp$ which preserves the arithmetic mean. By the first part of the proof we obtain the desired conclusion  (in fact, we also need to use the fact that any Jordan *-isomorphism respects any real powers of positive definite elements, see the last one among the above listed properties of Jordan *-isomorphisms).
\end{proof}

The proof of the analogous result concerning the Kubo-Ando power mean is much more complicated. In fact, before presenting it we need some more preparations. Firstly, although we have already used this concept implicitly above, recall the definition of the usual order among self-adjoint elements $A,B$ in a $C^*$-algebra. We write $A\leq B$ if $B-A$ is positive semidefinite. The strict order on the positive definite cone $\App$ of the $C^*$-algebra $\A$ is defined as follows: for $A,B\in \App$ we write $A<B$ if $B^\mfel AB^\mfel$ has spectrum contained in the open interval $]0,1[$. 

Secondly, we need to recall the notion of the Thompson metric $d_T$ which is defined on the positive definite cone $\App$ of the $C^*$-algebra $\A$ as follows:
\begin{equation}\label{E:Tho}
d_T (A, B) = \log \max \{M (A/B),M (B/A)\}, \quad A,B\in \mathcal A^{++},
\end{equation}
where $M(X/Y)= \inf\{ t>0 \, : \, X\leq t Y\}$ for any $X,Y\in \mathcal A^{++}$. It is easy to see (cf. \cite{ML09g}) that $d_T$ can also be rewritten as
\[
d_T (A,B)=\left\| \log \left(A^{-\frac{1}{2}} B A^{-\frac{1}{2}}\right) \right\|, \quad A,B\in \mathcal A^{++}.
\]
The structure of surjective Thompson isometries between positive definite cones of $C^*$-algebras is known, it was described in our paper \cite{ML14a}. Theorem 9 in \cite{ML14a} says that for given $C^*$-algebras $\mathcal A,\mathcal B$ and surjective map $\phi:\mathcal A^{++}\to \mathcal B^{++}$ we have that $\phi$ is a surjective Thompson isometry (i.e., a surjective isometry with respect to the metric $d_T$) if and only if
there are a central projection $P$ in $\mathcal B$ and a Jordan *-isomorphism $J:\mathcal A\to \mathcal B$ such that $\phi$ is of the form
\begin{equation}\label{E:36}
\phi(A) = \phi(I)^{1/2}\left(PJ(A)+(I-P)J(A^{-1})\right)\phi(I)^{1/2}
, \quad A\in \mathcal A^{++}.
\end{equation}
Here and in what follows, by a projection we mean a self-adjoint idempotent element which is called central if it commutes with any other element of the algebra.

Finally, we need a property called transfer property of Kubo-Ando means, which is related to the inequality (c) in the introduction. Namely, it is known that for an arbitrary Hilbert space  $H$ and for any Kubo-Ando mean $\sigma$, we have
\begin{equation}\label{E:transfer}
T(A\sigma B)T^*=(TAT^*) \sigma (TBT^*), \quad A,B\in \bhp
\end{equation}
for every invertible bounded either linear or conjugate linear operator $T$ on $H$. 

After this preparation, we can present the proof of our second theorem.

\begin{proof}[Proof of Theorem \ref{T:1}]
Assume that $\phi:\App \to \Bpp$ is a continuous bijective map
which satisfies
\begin{equation*}\label{E:1}
\phi(A\kamp B)=\phi(A)\kamp \phi(B), \quad A,B\in \App.
\end{equation*}
We can and do assume that $p$ is positive. Indeed, it is easy to see that
$A \mathfrak{m}_{-p} B= (A^{-1} \kamp B^{-1})^{-1}$ holds for any $A,B\in \App$ and then one can verify that the continuous bijective map $\psi:\App \to \Bpp$ defined by $\psi(A)=\phi(A^{-1})^{-1}$, $A\in \App$ satisfies 
\begin{equation*}\label{E:1}
\psi(A \mathfrak{m}_{-p} B)=\psi(A)\mathfrak{m}_{-p} \psi(B), \quad A,B\in \App.
\end{equation*}
We show that for any sequence $(B_n)$ in $\App$ such that $B_n\to 0$ in norm, we have that $\phi(B_n)\to 0$ in norm.
Assume that $B_n\to 0$ in norm. Then, by the formula \eqref{E:KA}, for any $A\in \App$ we have $A\kamp B_n\to A/2^\ep$, and hence we obtain that 
\[
\phi(A)\kamp \phi(B_n)=\phi(A\kamp B_n)\to \phi(A/2^\ep).
\]
It follows that $(\phi(B_n))$ is a sequence in $\Bpp$ with the property that for every $X\in \Bpp$, the sequence $(X\kamp \phi(B_n))$ is norm convergent in  $\B$. Choosing $X=I$, it follows again from \eqref{E:KA} that $(\phi(B_n))$ is norm convergent, let its limit be $C\in \Bp$. We have that
\[
\phi(A)\kamp C=\phi(A/2^\ep).
\]
Since, by the monotonicity of Kubo-Ando means (see (b) in the introduction), $X\kamp C\geq 0\kamp C= C/2^\ep$ for every $X\in \Bpp$, we deduce that 
\[
\phi(A/2^\ep)=\phi(A)\kamp C\geq C/2^\ep, \quad A\in \App
\]
meaning that $C$ is majorized by all elements of $\Bpp$. This implies that $C=0$, hence we infer that $\phi(B_n)\to 0$ in norm. From 
\[
\phi(A/2^\ep)=\phi(A)\kamp C=\phi(A)\kamp 0= \phi(A)/2^\ep, \quad A\in \App
\]
we obtain that $\phi$ respects the multiplication of elements by the constant $(1/2)^\ep$.

Assume that, for some given positive numbers $\alpha,\beta$, the equalities $\phi(\alpha^\ep A)=\alpha^\ep \phi(A)$ and $\phi(\beta^\ep A)=\beta^\ep \phi(A)$ hold for all $A\in \App$. Then, since for commuting positive definite elements the conventional and the Kubo-Ando power means coincide (see \eqref{E:com}), we easily compute
\[
\begin{gathered}
\phi((\alpha+\beta)^\ep A)/2^\ep=\phi(((\alpha+\beta)^\ep A)/2^\ep)=\phi((\alpha^\ep A)\cmp (\beta^\ep A))=
\phi((\alpha^\ep A)\kamp (\beta^\ep A))\\=\phi(\alpha^\ep A)\kamp \phi(\beta^\ep A)=(\alpha^\ep \phi(A))\kamp (\beta^\ep \phi(A))=(\alpha^\ep \phi(A))\cmp (\beta^\ep \phi(A))=
(\alpha+\beta)^\ep \phi(A)/2^\ep.
\end{gathered}
\]
This gives us that
\[
\phi((\alpha+ \beta)^\ep A)=(\alpha+\beta)^\ep \phi(A)
\]
holds for all $A\in \App$. From this observation we can easily 
deduce that $\phi(r^\ep A)=r^\ep \phi(A)$ holds for any positive rational number $r$ and element $A\in \App$. By the continuity of $\phi$, this implies that $\phi$ is positive homogeneous.

We next prove that $\phi:\App \to \Bpp$ is a strict order isomorphism meaning that for any $A,B\in \App$ we have
\[
A<B\Longleftrightarrow \phi(A)<\phi(B).
\]
In order to show this, observe that for any given $A,B\in \App$ there exists an $X\in \App$ such that $A\kamp X=B/2^\ep$ if and only if $A<B$. Indeed, using the transfer property and aplying trivial algebraic manipulations, $A\kamp X=B/2^\ep$ is equivalent to $I+(A^\mfel XA^\mfel)^p=(A^\mfel BA^\mfel)^p$ which has a solution $X\in  \App$ if and only if $I<(A^\mfel BA^\mfel)^p$. This latter inequality is equivalent to $A<B$. To see this, one may need to refer to the well-known fact that the spectrum of $A^\fel B^{-1} A^\fel$ equals that of $B^\mfel A B^\mfel$ (in fact, those two elements are unitarily equivalent as one can easily show by using the polar decomposition of $B^{-\fel} A^\fel$). 
Using the above characterization of strict order, it follows that $\phi$ is a strict order isomorphism.

Next, it is clear that 
\[
\inf\{ t>0\, :\, A\leq tB\}= \inf\{ t>0\, :\, A< tB\}
\]
holds for any $A,B\in \App$. It then follows from the positive homogeneity of $\phi$ and from the property that $\phi$ is a strict order isomorphism that $\phi$ is a positive homogeneous Thompson isometry from $\App$ onto $\Bpp$. By the structure \eqref{E:36} of surjective Thompson isometries, we obtain that $\phi$ is of the form $\phi(A)=DJ(A)D$, $A\in \App$ with some Jordan *-isomorphism $J:\A \to \B$ and element $D\in \Bpp$. (Indeed, the part $(I-P)J(A^{-1})$ in \eqref{E:36} must be missing due to the fact that the inverse operation is not homogeneous.) This gives us the sufficiency part of the theorem.

As for the converse statement, we argue as follows. By the above listed properties of Jordan *-isomorphisms, every such transformation is a continuous bijective map between positive definite cones which preserves all Kubo-Ando means and, by the transfer property, the same is true for any map $B\mapsto TBT^*$, $B\in \Bpp$ with $T\in \B$ invertible. The composition of two mean preserving maps is again mean preserving and hence we obtain the desired converse statement.
\end{proof}

For the proof of Theorem \ref{T:3},
recall that for any $A,B\in \App$ we have $A<B$ if and only if the spectrum of $B^\mfel AB^\mfel$ is contained in the open unit interval $]0,1[$. One can see that $A<B$ implies that $TAT^*<TBT^*$ for any invertible $T\in \A$. Indeed, this follows from the fact that $(TBT^*)^\mfel (TAT^*) (TBT^*)^\mfel$ is unitarily similar to $B^\mfel AB^\mfel$ (see the argument given in the first half of page 323 in \cite{MolGMU}).

\begin{proof}[Proof of Theorem \ref{T:3}]
Let $p\in]-1,1[$ and assume that the continuous bijective map $\phi:\App \to \Bpp$ satisfies
\begin{equation*}
\phi(A\cmp B)=\phi(A)\kamp \phi(B), \quad A,B\in \App.
\end{equation*}
In the first few steps of the proof we can closely follow the argument given in the proof of Theorem \ref{T:1}. Indeed, we may assume just like there that $p$ is positive. Next we can show in a very similar way that for any sequence $(B_n)$ in $\App$ which converges to 0 in norm, we have $\phi(B_n)\to 0$ in norm and also that $\phi(A/2^\ep)=\phi(A)/2^\ep$ holds for all $A\in \App$. After this, just as in the mentioned proof, we can verify that $\phi$ is positive homogeneous.

In the last part we observe that the equality $A\cmp X=B/2^\ep$ has a solution $X$ in $\App$ if and only if $A^p<B^p$ while, as we have seen in the proof of Theorem \ref{T:1}, $\phi(A)\kamp \phi(X)=\phi(B)/2^\ep$ has a solution if and only if $\phi(A)<\phi(B)$.  For the bijective transformation $\psi:\App \to \Bpp$ defined by $\psi(A)=\phi(A^\ep)$, $A\in \App$ we have
\[
A<tB \Longleftrightarrow \psi(A) <t^\ep\psi(B)
\]
for any $A,B\in \App$ and positive real number $t$. From this we obtain for the Thompson distances that 
\[
d_T(\psi(A),\psi(B))=\ep d_T(A,B), \quad A,B\in \App,
\]
see the definition of Thompson metric given in \eqref{E:Tho}.
This means that $\psi$ is a dilation (or, in other words, homothety) between the positive definite cones $\App$ and $\Bpp$.
We proved in Theorem 18 in \cite{MLxxx} that the existence of a non-isometric dilation between the positive definite cones of $C^*$-algebras implies that the underlying algebras are  necessarily commutative. This completes the proof of the statement.
\end{proof}

We now turn to the proofs of our results concerning maps on the positive semidefinite cone of a full operator algebra over a Hilbert space.

\begin{proof}[Proof of Theorem \ref{T:5}]
Let $\phi$ be as in the statement of the theorem.
For positive $p$, we clearly have
\[
A\cmp B=(A^p \mathfrak{m}_{1}  B^p)^{\frac{1}{p}}, \quad A,B\in \bhp.
\]
It follows that the bijective map $\psi:\bhp \to \bhp$ defined by $\psi(A)=\phi(A^\ep)^p$, $A\in \bhp$ satisfies 
\[
\psi(A \mathfrak{m}_{1}  B)=\psi(A) \mathfrak{m}_{1}  \psi(B), \quad A,B\in \bhp.
\]
Similarly, for negative $p$, by \eqref{E:8} it follows that the bijective map $\psi:\bhp \to \bhp$ defined by $\psi(A)=\phi(A^{\frac{1}{|p|}})^{|p|}$, $A\in \bhp$ satisfies
\[
\psi(A \mathfrak{m}_{-1}  B)=\psi(A) \mathfrak{m}_{-1}  \psi(B), \quad A,B\in \bhp.
\]
The structures of those maps are known. By the Theorem and Proposition in \cite{ML09f}, we have in both cases that there is an invertible bounded either linear or conjugate linear operator $T:H\to H$ such that $\psi(A)=TAT^*$, $A\in \bhp$. This completes the proof of the necessity part of the statement. The sufficiency follows immediately from the transfer property of the Kubo-Ando means $\mathfrak{m}_{1}$ and $\mathfrak{m}_{-1}$. 
\end{proof}

To prove Theorem \ref{T:4}, we need some auxiliary results that we present below.

\begin{lemma}\label{L:1}
Assume that $0<p\leq 1$. 
\begin{itemize}
\item[(i)]
If $A,B\in \bhp$ are commuting, then 
\[
A\kamp B=\ler{\frac{A^p+B^p}{2}}^\ep.
\]
\item[(ii)] 
For an arbitrary $A\in \bhp$, we have that $A=0$ if and only if for any $X,Y\in \bhp$, the equality $A=X\kamp Y$ implies $X=Y=A$.
\item[(iii)] For any $A\in \bhp$, denote 
\begin{equation*}
\mathfrak I(A)=
\{
(\ldots((A\kamp X_1)\kamp X_2)\ldots )\kamp X_n\, :\, n\in \mathbb N, X_1,\ldots , X_n\in \bhp\}.
\end{equation*}
The operator $B\in \bhp$ is invertible if and only if $\mathfrak I(B)$ is the minimum of the set $\{ \mathfrak{I}(A)\, :\, A\in \bhp\}$ partially ordered by the relation of inclusion.
\end{itemize}
\end{lemma}

\begin{proof}
The statement (i) is obvious for invertible $A,B\in \bhp$. For general $A,B\in \bhp$, consider $A+\epsilon I, B+\epsilon I$ for $\epsilon>0$, let $\epsilon$ tend to 0 monotone decreasingly, and use the continuity property of Kubo-Ando means (see (d) in the introduction).

As for (ii), assume first that $A=0$ and $A=X\kamp Y$ holds for $X,Y\in \bhp$. Then from $$X/2^\ep= X\kamp 0\leq X\kamp Y=0$$ we obtain that $X=0$, and in the same way we deduce that $Y=0$ also holds. 
Assume now that $A\neq 0$. Considering the spectral measure corresponding to $A$, there is a positive real number $s$ such that with the spectral measure $P$ of the set $[s,\infty[$ we have 
$P\neq 0$ and $sP\leq A$. Choosing any positive number $t<s$, we have that $tP\leq A$, $tP\neq A$. Let $X=tP$ and $Y=(2 A^p-X^p)^\ep$ which is a positive operator that commutes with $X$. It follows that $A=X\kamp Y$, $X\neq A$. This proves (ii).

To verify (iii), first assume that $B\in \bhp$ is invertible. Since $B\kamp X\geq B\kamp 0=B/2^\ep$, it follows that the elements of $\mathfrak I(B)$ are all invertible. On the other hand, let $C\in \bhp$ be invertible. Then taking the $\kamp$ mean of $B$ with 0-s sufficiently many times, we obtain an element $B'$ of $\mathfrak I(B)$ for which $B'\leq C$. We assert that then there exists $X\in \bhp$ such that $B'\kamp X=C/2^\ep$.
In fact, by the transfer property, the solvability of this latter equation is equivalent to that of
\[
(C^{\mfel}B'C^{\mfel})\kamp (C^{\mfel}XC^{\mfel})=I/2^\ep.
\]
With $B''=C^{\mfel}B'C^{\mfel}$, this is further equivalent to 
\[
B''\kamp Y=I/2^\ep
\]
for some $Y\in \bhp$. Since $B''\leq I$, choosing $Y=(I-B''^p)^\ep$, we have a solution $Y$ of the equality $B''\kamp Y=I/2^\ep$. Therefore, we obtain that there does exist $X\in \bhp$ such that $B'\kamp X=C/2^\ep$. This shows that $\mathfrak I(B)=\bhpp$ for any invertible $B\in \bhp$.

Let now $B\in \bhp$ be noninvertible. Then for large enough $t>0$, the element $B\kamp (t I-B^p)^\ep$ is a positive scalar multiple of the identity meaning that $\mathfrak I(B)$ contains one and then all invertible elements of $\bhp$ (see the argument in the previous paragraph). Since $B$ is noninvertible, we obtain that $\mathfrak I(B)$ contains $\bhpp$ as a proper subset. Therefore, $\mathfrak I(B)$ is not minimum, a contradiction.
The statement in (iii) now follows. 
\end{proof}

Assume next that $-1\leq p<0$. Then, denoting $q=-p$, the generating function of the symmetric Kubo-Ando mean $\kamp$ is the function
\[
f(t)=\ler{\frac{1+t^p}{2}}^\ep=
\ler{\frac{2t^q}{1+t^q}}^{\frac{1}{q}}, \quad t>0.
\]
Since this $f$ has limit 0 at 0, several important observations from the paper \cite{ML11g} can be applied. 
For the next lemma whose statements follow from those observation in \cite{ML11g}, we need to recall the following quantity which was originally introduced in 
\cite{BusGud99} under the name 'strength along a ray'.

Let $A\in \bhp$, consider a unit vector $\varphi$ in $H$ and denote by $P_\varphi$ the rank-one projection onto the subspace generated by $\varphi$ (recall that a rank-one operator is a bounded linear operator whose range is one-dimensional). The quantity
\begin{equation*}\label{E:strength}
\lambda(A, P_\varphi)= \sup \{ \lambda \geq 0 \, : \, \lambda
P_\varphi \leq A\}
\end{equation*}
is called the strength of $A$ along the ray represented by
$\varphi$. For curiosity, we mention that there is a nice and very useful formula for this quantity proved in \cite{BusGud99}:
\begin{equation*}\label{E:form2}
\lambda(A, P_\varphi)=
\left\{%
\begin{array}{ll}
    \|A^{-1/2}\varphi\|^{-2}, & \hbox{{\text{if }} $\varphi \in \rng (A^{1/2})$;} \\
    0, & \hbox{{\text{else}}.} \\
\end{array}%
\right.
\end{equation*}
(The symbol $\rng$ denotes the range of operators, and $A^{-1/2}$
stands for the inverse of $A^{1/2}$ on its range.)

After this, our next lemma reads as follows.

\begin{lemma}\label{L:2}
Assume $-1\leq p<0$. 
\begin{itemize}
\item[(i)]
For $A\in \bhp$, we have $A=0$ if and only if $A\kamp X=A$ holds for all $X\in \bhp$.
\item[(ii)] 
If $A\in \bhp$, we have that $A$ is invertible if and only if $\mathfrak I(A)=\bhp$. 
\item[(iii)]
The operator $A\in \bhp$ is a projection if and only if $I\kamp A=A$.
\item[(iv)] For any $A,B\in \bhp$, we have $A\kamp B= 0$ if and only if $\rng A^\fel \cap \rng B^\fel= \{0\}$.
\item[(v)] If $A\in \bhp$ is arbitrary and $P\in \bhp$ is a rank-one projection, then 
\[
A\kamp P=\ler{\frac{2\l(A,P)^q}{1+\l(A,P)^q}}^{\frac{1}{q}} P.
\]
\end{itemize}
\end{lemma}

\begin{proof}
To see (i), observe that $0\kamp X=X\kamp 0=0$ holds for all $X\in \bhp$ (this follows from the property $\lim_{t\to 0} f(t)=0$ mentioned three paragraphs before the formulation of the lemma). Conversely, if $A\kamp X=A$ for all $X\in \bhp$, then choosing $X=0$, we have $A=A\kamp 0=0$.
The statements in (ii)-(v) are particular cases of the statements in Lemmas 2.5, 2.2, 2.7, 2.6 in \cite{ML11g}, respectively. 
\end{proof}

We are now in a position to prove Theorem \ref{T:4}.

\begin{proof}[Proof of Theorem \ref{T:4}]
To begin with, first observe that the converse statement in the theorem, i.e., the fact that any map of the form $A\mapsto TAT^*$ on $\bhp$ (where $T$ is an invertible bounded either linear or conjugate linear operator on $H$) satisfies \eqref{E:7}, follows from the transfer property of Kubo-Ando means, see \eqref{E:transfer}.
 
Let now $\phi:\bhp \to \bhp$ be a bijective map such that
\begin{equation*}
\phi(A\kamp B)=\phi(A)\kamp \phi(B), \quad A,B\in \bhp.
\end{equation*}

First suppose that the number $p$ is positive. By (ii) in Lemma \ref{L:1}, we obtain that $\phi(0)=0$. It follows that
\[
\phi(A/2^\ep)=\phi(A\kamp 0)=\phi(A)\kamp 0=\phi(A)/2^\ep, \quad A\in \bhp.
\]
By (iii) in Lemma \ref{L:1}, $\phi$ maps $\bhpp$ onto itself. 
We assert that $\phi$ restricted to $\bhpp$ is an order automorphism of $\bhpp$. This will follow from the following observation: for any $A\in \bhp$ and $B\in \bhpp$ we have $A\leq B$ if and only if there is an $X\in \bhp$ such that $A\kamp X=B/2^\ep$. The sufficiency part of this characterization is apparent since
\[
A/2^\ep=A\kamp 0\leq A\kamp X=B/2^\ep.
\]
The converse implication, i.e., the necessity part can be proved as in the first part of the proof of (iii) in Lemma \ref{L:1}. Therefore, we obtain that for any $A\in \bhp$ and $B\in \bhpp$, the inequality $A\leq B$ holds if and only if $\phi(A)\leq \phi(B)$.

The structure of order automorphisms of $\bhpp$ (without assuming any sort of homogeneity) is known. By Theorem 1 in \cite{ML11a} we have that there is an invertible bounded either linear or conjugate linear operator $T:H\to H$ such that
$
\phi(B)=TBT^*
$
holds for all $B\in \bhpp$. Since we have already proved that for any $A\in \bhp$ and $B\in \bhpp$, the inequality $A\leq B$ is equivalent to $\phi(A)\leq \phi(B)$, it is just routine to verify that 
\[
\phi(A)=TAT^*
\]
holds for all $A\in \bhp$. Indeed, for any $A\in \bhp$ and $B\in \bhpp$ we have 
$A\leq B$ if and only if $T^{-1}\phi(A){T^*}^{-1}\leq B$ which easily implies that $T^{-1}\phi(A){T^*}^{-1}=A$.

Assume now that $p<0$. By (i) in Lemma \ref{L:2}, it follows that $\phi(0)=0$. Applying (ii) of the same lemma, we obtain that $\phi$ maps $\bhpp$ onto itself. Therefore, $\phi(I)$ is invertible.
Considering the map $\phi(I)^\mfel \phi(.)\phi(I)^\mfel$, we can and do assume that $\phi$ sends $I$ to $I$. Then, by (iii) in Lemma \ref{L:2}, $\phi$ preserves the projections in both directions: $P\in \bhp$ is a projection if and only if $\phi(P)$ is a projection.

We next see that for any $A,B\in \bhp$ we have $\rng A^\fel \leq \rng B^\fel$ if and only if  $\rng \phi(A)^\fel \leq \rng \phi(B)^\fel$. This can easily be deduced from the following observation: $\rng A^\fel \leq \rng B^\fel$ if and only if $B\kamp X=0$ implies $A\kamp X=0$ for any $X\in \bhp$, see (iv) in Lemma \ref{L:2}.

It follows that on the set of all projections on $H$, the map $\phi$ preserves the range inclusion in both directions implying that it is an order automorphism. Consequently, $\phi$ (and also $\phi^{-1}$) sends rank-one projections to rank-one projections and then we can infer that it sends rank-one elements in $\bhp$ to rank-one elements (again by the range inclusion preserving property). Referring to the fact that the positive semidefinite rank-one operators are exactly the nonnegative scalar multiples of rank-one projections,  we also easily obtain that for a given rank-one projection $P$, there is a bijective function $g$ of the nonnegative reals such that 
\[
\phi(tP)=g(t)\phi(P), \quad t\geq 0.
\]
In what follows we prove that $g$ is necessarily the identity.
For any positive real numbers $t,s$, we use the transfer property and (v) in Lemma \ref{L:2} to compute
\begin{equation}\label{E:11}
(tP)\kamp (sP)=s\ler{\ler{\frac{t}{s}}P\kamp P}=s\ler{\frac{2\ler{\frac{t}{s}}^q}{1+\ler{\frac{t}{s}}^q}}^\frac{1}{q}P,
\end{equation}
where $q=-p$.
We infer
\begin{equation*}
\begin{gathered}
g\ler{s\ler{\frac{2\ler{\frac{t}{s}}^q}{1+\ler{\frac{t}{s}}^q}}^\frac{1}{q}}\phi(P)=
\phi((tP)\kamp (sP))=\phi(tP)\kamp\phi(sP)\\=(g(t)\phi(P))\kamp (g(s)\phi(P))=
g(s)\ler{\frac{2\ler{\frac{g(t)}{g(s)}}^q}{1+\ler{\frac{g(t)}{g(s)}}^q}}^\frac{1}{q}\phi(P).
\end{gathered}
\end{equation*}
We therefore obtain the equality 
\begin{equation*}
\begin{gathered}
g\ler{s\ler{\frac{2\ler{\frac{t}{s}}^q}{1+\ler{\frac{t}{s}}^q}}^\frac{1}{q}}=
g(s)\ler{\frac{2\ler{\frac{g(t)}{g(s)}}^q}{1+\ler{\frac{g(t)}{g(s)}}^q}}^\frac{1}{q}.
\end{gathered}
\end{equation*}
Denoting $h(t)=g(t^{\frac{1}{q}})^q$, $t\geq 0$, this apparently implies
\begin{equation*}
h\ler{\frac{2t^qs^q}{t^q+s^q}}=\frac{2h(t^q)h(s^q)}{h(t^q)+h(s^q)}
\end{equation*}
and then that
\begin{equation*}
h\ler{\frac{2ts}{t+s}}=\frac{2h(t)h(s)}{h(t)+h(s)}.
\end{equation*}
Denoting $k(t)=1/h(1/t)$, $t>0$, we have 
\begin{equation}\label{E:12}
k\ler{\frac{t+s}{2}}=\frac{k(t)+k(s)}{2}
\end{equation}
for any positive real numbers $t,s$.
Since $k$ is a bijection of the positive half-line (which follows from the same property of $g$), we deduce that $k$ is a constant multiple of the identity (cf. Theorem \ref{T:2}). By $g(1)=1$ we thus obtain that $k$ is the identity and this implies that $g$ is also the identity, which was our claim.

We can now verify that $\phi$ is an order automorphism of $\bhp$.
Indeed, from \cite{BusGud99} we know that for any $A,B\in \bhp$, the inequality $A\leq B$ holds if and only if $\l(A,P)\leq \l(B,P)$ holds for all rank-one projections $P\in  \bhp$.
The equation $\phi(A\kamp P)=\phi(A)\kamp \phi(P)$ implies the identity
\begin{equation*}
\ler{\frac{2\l(A,P)^q}{1+\l(A,P)^q}}^{\frac{1}{q}}=\ler{\frac{2\l(\phi(A),\phi(P))^q}{1+\l(\phi(A),\phi(P))^q}}^{\frac{1}{q}}
\end{equation*}
from which we easily obtain that $\l(A,P)\leq \l(B,P)$ if and only if $\l(\phi(A),\phi(P))\leq \l(\phi(B),\phi(P))$. We thus infer that $\phi$ is an order isomorphism of $\bhp$. The structure of those maps was determined in \cite{ML11a}.  It follows from the results there, that we have an invertible bounded either linear or conjugate linear operator $T$ on $H$ such that $\phi(A)=TAT^*$, $A\in \bhp$. This completes the proof of our theorem.
\end{proof}

Finally, we present the proof of the last result of the paper.

\begin{proof}[Proof of Theorem \ref{T:6}]
Contrary to the assertion, assume that $\phi:\bhp \to \bhp$ is a bijective map such that
\begin{equation}\label{E:10}
\phi(A\cmp B)=\phi(A)\kamp \phi(B), \quad A,B\in \bhp.
\end{equation}
Assume first that $p$ is positive. 
It is easy to see that the characterization of 0 given in (ii) in Lemma \ref{L:1} as well as the characterization of invertibility given in (iii) are valid also for the conventional power mean $\cmp$.
It follows that $\phi(0)=0$ implying that $\phi$ satisfies $\phi(A/2^\ep)=\phi(A)/2^\ep$ for all $A\in \bhp$, and also that $\phi$ maps $\bhpp$ onto itself. 

As we have seen in the proof of Theorem \ref{T:4},
for any $A,B\in \bhpp$, the solvability of the equation $A\kamp X=B/2^\ep$ for $X\in \bhp$ is equivalent to the inequality $A\leq B$ while the solvability of the equation $A\cmp X=B/2^\ep$ is clearly equivalent to $A^p\leq B^p$. It follows that the map $A\mapsto \phi(A^\ep)$ is an order automorphism of $\bhpp$. As above, we conclude that there is an invertible bounded linear or conjugate linear operator $T$ on $H$ such that $\phi(A)=TA^pT^*$, $A\in \bhpp$. Using \eqref{E:10}, we easily obtain that
\[
T\frac{A^p+B^p}{2}T^*=(TA^p T^*)\kamp (TB^pT^*) =T(A^p\kamp B^p)T^*, \quad A,B\in \bhpp,
\]
and then that
\[
\frac{A+B}{2}=A\kamp B, \quad A,B\in \bhpp.
\]
This immediately implies that $p=1$, a contradiction.

Assume now that $p$ is negative. Recall the definition of the conventional power mean $\cmp$ given in \eqref{E:8}. 
It is very easy to see that the characterizations of 0, invertible elements and projections given in (i), (ii) and (iii) in Lemma \ref{L:2} are valid also for $\cmp$, but (iv) and (v) change as follows.
For any $A,B\in \bhp$, we have $A\cmp B= 0$ if and only if $\rng A^{\frac{q}{2}} \cap \rng B^{\frac{q}{2}}= \{0\}$, and for an arbitrary $A\in \bhp$ and rank-one projection $P\in \bhp$, we have
\[
A\cmp P=\ler{\frac{2\l(A^q,P)}{1+\l(A^q,P)}}^{\frac{1}{q}} P.
\]
Here, $q=-p$.
After this, following the second part of the proof of Theorem \ref{T:4}, we infer that $\phi(I)$ is invertible and hence the map $\phi(I)^\mfel \phi(.)\phi(I)^\mfel$ also satisfies \eqref{E:10}, moreover it sends $I$ to $I$. Therefore, we can clearly assume that the original transformation $\phi$ already has this additional property. It then follows that $\phi$ sends projections to projections and we can continue following the argument in the proof of Theorem \ref{T:4} till the point that for any rank-one projection $P$ on $H$, there is a bijective function $g$ of the nonnegative reals such that
\[
\phi(tP)=g(t)\phi(P), \quad t\geq 0
\]
holds.
Observe that for commuting $A,B\in \bhp$, we have $A\kamp B=A\cmp B$. In fact, for invertible $A,B$ this is easy to check and then we can apply the continuity property of Kubo-Ando means (see the explanation in the introduction of $\cmp$ for negative $p$ presented before the formulation of Theorem \ref{T:5}). Therefore, the same computation as from \eqref{E:11} to \eqref{E:12} applies and we obtain that 
$g$ is the identity.

If $A\in \bhp$ is arbitrary and $P\in \bhp$ is any rank-one projection, then
from $\phi(A\cmp P)=\phi(A)\kamp \phi(P)$ we deduce that
\[
\ler{\frac{2\l(A^q,P)}{1+\l(A^q,P)}}^{\frac{1}{q}} \phi(P)=
\ler{\frac{2\l(\phi(A),\phi(P))^q}{1+\l(\phi(A),\phi(P))^q}}^{\frac{1}{q}} \phi(P).
\]
This implies that we have $\l(A^q,P)\leq \l(B^q,P)$ if and only if $\l(\phi(A),\phi(P))\leq \l(\phi(B),\phi(P))$. For any $A,B\in \bhp$, we infer that $A^q\leq B^q$ if and only if $\phi(A)\leq \phi(B)$. This means that the map $A\mapsto \phi(A^{\frac{1}{q}})$ is an order isomorphism of $\bhp$. As before, this implies that there is an invertible bounded either linear or conjugate linear operator $T$ on $H$ such that
\[
\phi(A)=TA^qT^*, \quad A\in \bhp.
\]
From \eqref{E:10} we obtain that 
\[
T(A\cmp B)^q T^*=(TA^qT^*)\kamp (TB^q T^*)
\]
or equivalently that
\[
(A\cmp B)^q=A^q\kamp B^q
\]
holds for all $A,B\in \bhp$.
It implies that
\[
A^q\mathfrak{m}_{-1} B^q=A^q\kamp B^q, \quad A,B\in \bhp
\]
and hence we can conclude that $p=-1$, a contradiction again.
This finishes the proof of the theorem.
\end{proof}

We close the paper with the following open problems.
It would be interesting to clarify if the continuity assumptions in Theorems \ref{T:1} and \ref{T:3} can be dropped. Furthermore,
it seems a really nontrivial problem to investigate how the statements in Theorems \ref{T:5}-\ref{T:6} survive in general $C^*$-algebras.
Finally, we find it also interesting for $p\neq \pm 1$ to consider the problem if there is any nontrivial homomorphism from $\bhp$ equipped with the operation $\kamp$ into $\bhp$ equipped with $\cmp$  (Theorem \ref{T:6} asserts that there is no such isomorphism).

\bibliographystyle{amsplain}

\end{document}